\documentclass[11pt]{amsart}

\pdfoutput=1

\usepackage{amsmath} 
\allowdisplaybreaks 
\usepackage[foot]{amsaddr}

\usepackage{appendix}
\usepackage{cancel}
\usepackage{esint,amssymb} 
\usepackage{graphicx}
\usepackage{mathtools} 
\usepackage{xcolor}
\usepackage[colorlinks=true, pdfstartview=FitV, linkcolor=blue, citecolor=blue, urlcolor=blue,pagebackref=false]{hyperref}
\usepackage{orcidlink}
\usepackage{microtype}

\usepackage{bm}
\usepackage{mathrsfs}
\usepackage{enumitem}

\definecolor{darkgreen}{rgb}{0,0.5,0}
\definecolor{darkblue}{rgb}{0,0,0.7}
\definecolor{darkred}{rgb}{0.9,0.1,0.1}

\newtheorem{proposition}{Proposition}
\newtheorem{theorem}[proposition]{Theorem}
\newtheorem{lemma}[proposition]{Lemma}

\theoremstyle{remark}
\newtheorem{remark}[proposition]{Remark}

\theoremstyle{definition}
\newtheorem{definition}[proposition]{Definition}

\numberwithin{equation}{section}
\numberwithin{proposition}{section}
\numberwithin{figure}{section}
\numberwithin{table}{section}

\newcommand{\N}{\mathbb{N}}
\newcommand{\Q}{\mathbb{Q}}
\newcommand{\R}{\mathbb{R}}

\newcommand{\E}{\mathbb{E}}
\renewcommand{\P}{\mathbb{P}}

\renewcommand{\leq}{\leqslant}
\renewcommand{\geq}{\geqslant}

\renewcommand{\subset}{\subseteq}
\renewcommand{\bar}{\overline}
\renewcommand{\tilde}{\widetilde}

\newcommand{\Ll}{\left}
\newcommand{\Rr}{\right}
\renewcommand{\d}{\mathrm{d}}

\newcommand{\mcl}{\mathcal}

\newcommand{\mfk}{\mathfrak}
\newcommand{\msc}{\mathscr}

\newcommand{\eps}{\varepsilon}

\newcommand{\la}{\left\langle}
\newcommand{\ra}{\right\rangle}

\renewcommand{\S}{\mathrm{S}}

\DeclareMathOperator{\supp}{supp}

\newcommand{\fR}{\mathfrak{R}}

\newenvironment{e}{\begin{equation}}{\end{equation}\ignorespacesafterend}
\newenvironment{e*}{\begin{equation*}}{\end{equation*}\ignorespacesafterend}

\begin{document}

\author{Hong-Bin Chen\,\orcidlink{0000-0001-6412-0800}}
\address[Hong-Bin Chen]{NYU-ECNU Institute of Mathematical Sciences, NYU Shanghai, China}
\email{\href{mailto:hongbin.chen@nyu.edu}{hongbin.chen@nyu.edu}}

\author{Victor Issa\,\orcidlink{0009-0009-1304-046X}}
\address[Victor Issa]{Department of Mathematics, ENS Lyon, France}
\email{\href{mailto:victor.issa@ens-lyon.fr}{victor.issa@ens-lyon.fr}}


\title[Parisi formula and large deviations with Mattis interactions]{Parisi formula and large deviations for mean-field spin glasses with Mattis interactions}

\begin{abstract}
We consider convex mean-field vector spin glasses perturbed by a general Mattis interaction $N G(m_N)$, where $m_N$ is a bounded vector-valued empirical average and $G$ is continuous. We prove a Parisi-type variational formula for the limiting quenched free energy and a quenched large deviation principle for $m_N$, both at every temperature. The approach is simple and modular: the only model-specific spin-glass input is the variational formula for every linear tilt of $m_N$ together with differentiability of the resulting free energy. The classical G\"artner--Ellis, Varadhan, and Bryc transfers then treat an arbitrary continuous $G$. Thus the same argument applies to a broad class of empirical observables and Mattis potentials, without a separate constrained-free-energy calculation for each model. As a scalar specialization, the framework contains the Gaussian spin-glass surrogate used in mismatched rank-one matrix estimation.


 \bigskip


    \medskip


\end{abstract}

\maketitle

\setcounter{tocdepth}{1}
\thispagestyle{empty}
{
  \hypersetup{linkcolor=black}
  \setcounter{tocdepth}{1}
  \tableofcontents
}


\section{Introduction}

\subsection{Preamble}

We study mean-field spin glasses perturbed by a Mattis interaction that depends on a bounded empirical order parameter. Our main results give, under the usual convexity condition~\eqref{e.convexity}, a Parisi-type formula for the limiting free energy and a quenched large deviation principle for the order parameter. Both conclusions hold at every temperature. The proof separates the model-specific Parisi calculation for linear tilts from a simple large-deviation transfer. This makes the method applicable to general bounded vector-valued empirical observables and arbitrary continuous Mattis potentials. 

A scalar example of the models we treat is the system of $\pm1$ spins with energy
\begin{equation} \label{e.simple}
    E_N(\sigma) = \frac{\beta}{\sqrt{N}} \sum_{i,j = 1}^N W_{ij} \sigma_i \sigma_j + \left( \frac{1}{\sqrt{N}}\sum_{i = 1}^N \chi_i \sigma_i \right)^2, \qquad \forall \sigma \in \{-1,1\}^N, 
\end{equation}
where the array $(W_{i,j})_{1 \leq i,j \leq N}$ consists of independent standard Gaussian variables and $(\chi_i)_{1 \leq i \leq N}$ are independent symmetric signs. Such Hamiltonians arise, in particular, as Gaussian surrogates of mismatched low-rank estimation problems \cite{barbier2025mismatchlinear,barbier2022priceofignorance,krzakala2016thresholds,macris2022mismatched,chen2014mixedSK,camilli2022mismatch,guionnet2025estimating}.

For the example~\eqref{e.simple}, we identify the large-$N$ limit of
\begin{equation*} 
    f_N = -\frac{1}{N} \log \left( \frac{1}{2^N} \sum_{\sigma \in \{-1,1\}^N} e^{E_N(\sigma)} \right),
\end{equation*}
and establish a large deviation principle for the Mattis magnetization
\begin{e*}
    m_N = \frac{1}{N} \sum_{i = 1}^N \chi_i \sigma_i
\end{e*}
under the corresponding Gibbs measure $\langle \cdot \rangle_N$, defined by
\begin{equation*}
 \langle g(\sigma) \rangle_N \propto  \frac{1}{2^N} \sum_{\sigma \in \{-1,1\}^N} g(\sigma) e^{E_N(\sigma)}.
\end{equation*}
In the general model, the scalar product $\chi_i\sigma_i$ is replaced by a bounded vector-valued observable $h(\sigma_{\bullet i},\chi_i)$, and the quadratic Mattis term is replaced by $N G(m_N)$ for an arbitrary continuous $G$.

For reference, in the scalar Gaussian surrogate of \cite[(2.5)]{guionnet2025estimating}, one takes
\begin{equation*}
    h(\tau,\chi)=(\tau^2,\tau\chi),
    \qquad
    G(S,M)=\frac{\beta_{\rm S}}{4}S^2+\frac{\beta_{\rm SNR}}{2}M^2.
\end{equation*}
After translating the free-energy sign convention and ignoring the uniformly bounded diagonal correction, Theorems~\ref{t.F_N(...)} and~\ref{t.LDP_gen} recover the Gaussian-surrogate conclusions in \cite[Theorems~2.6 and~2.7]{guionnet2025estimating}. The universality reduction from general output channels to that surrogate, \cite[Proposition~2.4]{guionnet2025estimating}, is a separate input and is not proved here.

Our main results are stated for very general models in Theorem~\ref{t.F_N(...)} and Theorem~\ref{t.LDP_gen}. Here we state a specialized version of those results for the model defined by \eqref{e.simple}, which is an instance of a model satisfying the convexity assumption \eqref{e.convexity} defined below. The definition of the function $\psi = \psi(q;x)$ appearing in this statement is given below in \eqref{e.psi(q;x)=}.
\begin{theorem}\label{t.parisi+ldp.basic}
    Almost surely over the randomness of $(W_{i,j})_{i,j}$ and $(\chi_i)_i$, we have
    \begin{e*}
        \lim_{N \to +\infty} f_N = \inf_{m} \sup_{x,p} \left\{ \psi(\beta^2 p;x) - \frac{\beta^2}{2} \int_0^1 p(s)^2 \d s + m x - m^2 - \frac{\beta^2}{2}\right\},
    \end{e*}
    where the infimum is taken over $m \in \R$ and the supremum over $x \in \R$ and $p : [0,1) \to \R_+$ bounded increasing càdlàg paths. Furthermore, still almost surely over the randomness of $(W_{i,j})_{i,j}$ and $(\chi_i)_i$, the random variable $m_N$ satisfies a large deviation principle under $\langle \cdot \rangle_N$ with rate function 
    \begin{e*}
        J(m) = -m^2 + \phi^*(m) + \sup_{m' \in \R} \left\{(m')^2 - \phi^*(m') \right\}.
    \end{e*}
    Here $\phi^*(m) = \sup_{x \in \R} \left\{xm + \phi(x)\right\}$ and $\phi(x)$ is the almost sure limit as $N \to +\infty$ of the free energy $f_N$ where the term $\left( \frac{1}{\sqrt{N}}\sum_{i = 1}^N \chi_i \sigma_i \right)^2$ in $E_N$ is replaced by $ x \sum_{i = 1}^N \chi_i \sigma_i$. 
\end{theorem}

Note that as a consequence of Theorem~\ref{t.F_N(...)} below, the quantity $\phi(x)$ used to define the rate function $J$ can be expressed as a variational formula. 

For convex covariance functions, the free energy without a Mattis interaction is described by the Parisi formula \cite{parisi1979infinite,gue03,Tpaper}; Hamilton--Jacobi formulations appear in \cite{guerra2001sum,mourrat2019parisi,mourrat2020free,HJbook}. Analyses of particular Mattis or planted models frequently proceed through constrained free energies of the form
\begin{equation*}
    f_N(A) = -\frac{1}{N} \log \left( \frac{1}{2^N} \sum_{\sigma \in A} e^{E_N(\sigma)} \right);
\end{equation*}
see, for example, \cite{guionnet2025estimating,franz2020largedeviations,camilli2022mismatch,pan.vec,chen2014mixedSK}.

In this paper, we adopt a different approach. We instead keep the continuous Mattis potential $G$ as a parameter and compute
\begin{equation*}
    f_N^G = -\frac{1}{N} \log \left( \frac{1}{2^N} \sum_{\sigma \in \{-1,1\}^N} \exp\left(\frac{\beta}{\sqrt{N}} \sum_{i,j = 1}^N W_{ij} \sigma_i \sigma_j + NG(m_N) \right) \right).
\end{equation*}
The spin-glass input is the Parisi formula for every linear tilt $G(m)=x\cdot m$ and the differentiability in $x$ of the resulting limit. Thereafter the transfer is classical: the Gärtner--Ellis theorem gives the LDP at $G=0$, Varadhan's lemma evaluates continuous tilts, and Bryc's inverse lemma gives the corresponding tilted LDP. In particular,
\begin{e*}
    \frac{1}{N} \log \langle e^{NG(m_N)}\rangle^0_N = -f_N^G + f_N^0.
\end{e*}
Once this is done, one can recover Theorem~\ref{t.parisi+ldp.basic} by setting $G(m) = m^2$.

\subsection{Setting} \label{ss.setting}
In this section, we give a definition of the class of models that we treat. 

For $m,n\in\N$, we denote by $\R^{m\times n}$ the space of all $m\times n$ real matrices. For any $a=(a_{ij})_{1\leq i\leq m,\, 1\leq j\leq n}\in \R^{m\times n}$, we denote its $j$-th column vector as $a_{\bullet j}=(a_{ij})_{1\leq i\leq m}$ and its $i$-th row vector as $a_{i\bullet}=(a_{ij})_{1\leq j\leq n}$. If not specified, a vector is understood to be a column vector. For a matrix or vector $a$, we denote by $a^\intercal$ its transpose.
For $a,b\in \R^{m\times n}$, we write $a\cdot b= \sum_{ij}a_{ij}b_{ij}$, $|a|=\sqrt{a\cdot a}$, and similarly for vectors. 
For $n\in \N$, let $\S^n$ be the linear space of $n\times n$ real symmetric matrices 
and also let $\S^n_+$ (resp.\ $\S^n_{++}$) be the subset consisting of positive semi-definite (resp.\ definite) matrices. For $a,b \in \S^n$, we write $a\geq b$ provided $a-b\in\S^n_+$, which gives a natural partial order on $\S^n$.

We start by defining the spin glass part of the model. Let $D\in\N$ be the dimension of a single spin. Let $P_1$ be a finite positive measure supported on the unit ball $\Ll\{\tau\in\R^D:\:|\tau|\leq1\Rr\}$ of $\R^D$ and we assume that the linear span of vectors in $\supp P_1$ is $\R^D$. 

For each $N\in\N$, we denote by $\sigma =(\sigma_{\bullet i})_{i=1}^N =(\sigma_{di})_{1\leq d\leq D, \, 1\leq i\leq N}\in\R^{D\times N}$ the spin configuration consisting of $N$ $\R^D$-valued spins $\sigma_{\bullet i}$ as column vectors. We view $\sigma$ as a $D\times N$ matrix. We sample each spin independently from $P_1$, and thus the distribution of $\sigma$ is given by
\begin{e*}
     \d P_1^{\otimes N}(\sigma)= \otimes_{i=1}^N \d P_1( \sigma_{\bullet i}).
\end{e*}
We let $\xi:\R^{D\times D}\to \R$ be a smooth function that can be expressed as a power series,
and we assume that
\begin{equation}\label{e.convexity} 
    \text{the function $\xi$ is convex on $S^D_+$.} 
\end{equation}
For each $N$, we consider a centered Gaussian field $(H_N(\sigma))_{\sigma\in\R^{D\times N}}$ with covariance structure given by
\begin{e}\label{e.cov}
    \E \Ll[H_N(\sigma)H_N(\sigma')\Rr] = N \xi \Ll( \tfrac{\sigma\sigma'^\intercal}{N}\Rr),\qquad\forall \sigma,\sigma'\in \R^{D\times N}.
\end{e}
We also define $\theta:\R^{D\times D}\to\R$ by
\begin{align}\label{e.theta=}
    \theta (a) = a\cdot\nabla\xi(a)-\xi(a),\quad\forall a\in\R^{D\times D}.
\end{align}

We now describe the Mattis interaction part of the model. Let $L\in\N$ and let $(\chi_i)_{i\in\N}$ be i.i.d.\ $\R^L$-valued random vectors, which are, in particular, independent from $H_N(\sigma)$. Let $d\in\N$ and $h:\R^D\times \R^L \to \R^d$ be a bounded\footnote{Throughout it will be enough to assume that $h$ is bounded on the support of the positive measure $P_1 \otimes \text{Law}(\chi_1)$. } measurable function. For each $i\in\N$, we view $h(\sigma_{\bullet i},\chi_i)$ as a generalized spin and define the mean magnetization
\begin{align}\label{e.m_N=}
    m_N = \frac{1}{N}\sum_{i=1}^N h(\sigma_{\bullet i},\chi_i).
\end{align}
Let $G:\R^d\to\R$ be a continuous function, and we add to the system the Curie--Weiss-type interaction $NG\Ll(m_N\Rr)$. 

At this point, we have described the spin glass part and the Mattis interaction part of our energy function. We now introduce the enriched formulation used to state the Parisi input uniformly over linear Mattis tilts. A reader interested only in the classical setting may throughout choose the parameter $q$ to be the constant path taking the value $0$ and set $t = \beta^2/2$.

To define the enriched version of the model, we need to add external fields parameterized by increasing paths and driven by Poisson--Dirichlet cascades.  Let $\mcl Q_\infty$ be the collection of bounded increasing càdlàg paths $q:[0,1) \to \S^D_+$, where the monotonicity is understood as $q(s)-q(s')\in\S^D_+$ for $s>s'$. Throughout, given $q \in \mcl Q_\infty$, we will define 
\begin{e*}
    q(1) = \lim_{s \to 1^-} q(s),
\end{e*}
which is well-defined by monotonicity of $q$.

Throughout, let $\fR$ be the Ruelle probability cascade (RPC) with overlap uniformly distributed over $[0,1]$ (see \cite[Theorem~2.17]{pan}). Precisely, $\fR$ is a random probability measure on the unit sphere of a fixed separable Hilbert space (the exact form of the space is not important), with the inner product denoted by $\alpha\wedge\alpha'$. Let $\alpha$ and $\alpha'$ be independent samples from $\fR$. Then, the law of $\alpha\wedge\alpha'$ under $\E \fR^{\otimes 2}$ is the uniform distribution over $[0,1]$, where $\E$ integrates the randomness of $\fR$.
This overlap distribution uniquely determines $\fR$ (see~\cite[Theorem~2.13]{pan}).
Almost surely, the support of $\fR$ is ultrametric in the induced topology. For rigorous definitions and basic properties, we refer to \cite[Chapter 2]{pan} (also see \cite[Chapter 5]{HJbook}).
We also refer to~\cite[Section~4]{chenmourrat2023cavity} for the construction and properties of $\fR$ useful in this work.

For $q\in\mcl Q_\infty$ and almost every realization of $\fR$, let $(w^q(\alpha))_{\alpha\in\supp\fR}$ be the $\R^D$-valued centered Gaussian process with covariance
\begin{align}\label{e.Ew^qw^q=}
    \E\Ll[ w^q (\alpha)w^q(\alpha')^\intercal\Rr]  = q(\alpha\wedge\alpha').
\end{align}
Its existence and properties are given in~\cite[Section~4]{chenmourrat2023cavity}.
Conditioned on $\fR$, let $(w^q_i)_{i\in\N}$ be i.i.d.\ copies of $w^q$ and also independent of other randomness.

For $N\in \N$, $q\in  \mcl Q_\infty$, $t \geq 0$, the full energy function of the system will be chosen to be 
\begin{align} \label{e.H^q,x_N(sigma,alpha,chi)=}
\begin{split}
    H^{G}_N(\sigma,\alpha) &= N G(m_N) + \sqrt{2t} H_N(\sigma) - Nt \xi\Ll(\tfrac{\sigma\sigma^\intercal}{N}\Rr) 
    \\
    &+ \sqrt{2}\sum_{i=1}^N w^q_i(\alpha)\cdot \sigma_{\bullet i} - q(1)\cdot\sigma\sigma^\intercal.   
\end{split} 
\end{align}
We are interested in the infinite-volume limit of the associated free energy defined as
\begin{align}\label{e.F_N(beta)}
    F_N^G = -\frac{1}{N} \log \iint \exp\Ll(H_N^{G}(\sigma,\alpha) \Rr) \d P_1^{\otimes N}(\sigma) \d \mfk R(\alpha).
\end{align}
In principle, $H_N^G$ and $F_N^G$ depend on the choice of $t$ and $q$. Throughout this paper, we will keep those parameters fixed so we do not make this dependence explicitly appear in our notations for $H_N^G$ and $F_N^G$.

As is explained in full detail in Section~\ref{ss.usual}, usual models such as \eqref{e.simple} can be obtained as special cases of the Hamiltonian $H_N^G$. Briefly, to do so one can set $t = \beta^2/2$ and $q = 0$ and make a wise choice for $G$ in order to absorb $- \tfrac{N\beta^2}{2}\xi\Ll(\tfrac{\sigma\sigma^\intercal}{N}\Rr)$ in the Mattis interaction.

Here, we have chosen to introduce the correction terms $Nt \xi\Ll(\sigma\sigma^\intercal/N\Rr)$ and $q(1)\cdot\sigma\sigma^\intercal$ as a convenience to simplify the forthcoming variational formulas. As explained in Section~\ref{ss.usual}, they are easily removed. These terms are one-half of the variances of $ \sqrt{2t}H_N(\sigma)$ and $\sqrt{2}\sum_{i=1}^Nw^q_i(\alpha)\cdot \sigma_{\bullet i}$, respectively, and therefore play the role of exponential-martingale drifts.

The large deviation principle we prove holds almost surely with respect to the randomness of $(H_N)_{N \geq 1}$, $(\chi_i)_{i \in \N}$, and $\mfk R$. Therefore, throughout we will assume that those random variables are defined on a common probability space $\Omega$ in order to have a convenient way of stating those results.
\subsection{Main results} 

We let $\bar{F}_N^G = \E F_N^G$, where $\E$ averages over the randomness of $H_N$, $(\chi_i)_i$, and $\mfk R$. The random variable $F_N^G$ concentrates strongly around its expectation $\bar{F}_N^G$ in the large $N$ limit \cite[Theorem~1.2]{pan}. Hence, to describe the limit as $N \to +\infty$ of $F_N^G$, it suffices to describe the limit of $\bar{F}_N^G$. To describe this limit, we need to define, for $q\in\mcl Q_\infty$ and $x\in \R^d$, 
\begin{align}
    &\psi(q;x)  \notag
    \\
    &= -\E\log\iint\exp\Ll(\sqrt{2}w^q(\alpha)\cdot \tau-q(1)\cdot\tau\tau^\intercal+x\cdot h(\tau,\chi_1)\Rr) \d P_1(\tau)\d\mathfrak{R}(\alpha), \label{e.psi(q;x)=}
\end{align}
where $\E$ integrates the randomness of $\mfk R, (\chi_i)_i$ and $w^q$ so that $\psi$ is a deterministic function.

\begin{theorem}\label{t.F_N(...)}
Recall that we assume \eqref{e.convexity}. For every continuous function $G : \R^d \to \R$, we have
\begin{align}\label{e.t.gen}
\begin{split}
    &\lim_{N\to\infty} \bar{  F}_N^G
    =\inf_{m}\sup_{x,\, p}\Ll\{\psi\Ll(q+t\nabla \xi(p);x\Rr)-t\int_0^1\theta(p(s))\d s + m \cdot x - G(m)\Rr\},
\end{split}
\end{align}
where $\inf$ is taken over $m \in\R^d$ and $\sup$ over $x\in\R^d$ and $p\in\mcl Q_\infty$.
\end{theorem}

When $G = 0$, \eqref{e.t.gen} boils down to the classical Parisi formula as written in~\cite[Proposition~8.4]{chenmourrat2023cavity}. Note that here and in~\cite[Proposition~8.4]{chenmourrat2023cavity} the convexity of $\xi$ is assumed on $S^D_+$ and not on the whole space $\R^{D \times D}$. This weaker convexity assumption should not be a surprise. Indeed, in the $D = 1$ case, it can be seen that the Parisi formula still holds assuming only that $\xi$ is convex on $\R_+$ as a consequence of Talagrand's positivity principle, and this weaker convexity assumption is simply a consequence of a generalization of this observation (but with a different proof).

Throughout the paper, the free energy of models with linear Mattis interaction will play a special role. So for every $x \in \R^d$, we let 
\begin{e} \label{e.def.phi}
     \varphi(x) = \lim_{N \to +\infty} \bar{ F}_N^{m \mapsto x \cdot m},
\end{e}
and we let $\varphi^*(m) = \sup_{x} \left\{ x \cdot m + \varphi(x) \right\}$ denote the convex dual of $-\varphi$.
\begin{definition}[Large deviation principle]
    Let $(X_N)_{N \geq 1}$ be a sequence of random variables in $\R^d$ with associated probability measures $(\P_N)_{N \geq 1}$ and let $I:\R^d\to [0,\infty]$. We say that the random variables $X_N$ satisfy a large deviation principle under $\P_N$ with rate function $I$ when for every Borel measurable set $A \subset \R^d$, we have  
    \begin{align*}
        -\inf_{ A^\circ} I
        &\leq \liminf_{N\to\infty}\frac{1}{N}\log \P_N\{X_N\in A\}
        \leq \limsup_{N\to\infty}\frac{1}{N}\log \P_N\{X_N\in A\} \\
        &\leq -\inf_{\bar A} I,
    \end{align*}
    where $A^\circ$ and $\bar A$ are the interior and the closure of $A$, respectively.
\end{definition}

In the following, we refer to the large deviation principle as LDP.

Let $\la\cdot\ra^{G}_{N}$ be the probability measure defined by 
\begin{equation*} 
    \la f(\sigma,\alpha) \ra^{G}_{N} \propto \iint f(\sigma,\alpha) e^{H^G_N(\sigma,\alpha)} \d P_1^{\otimes N}(\sigma)\d\mathfrak R(\alpha).
\end{equation*}
Observe that $\la\cdot\ra^{G}_{N}$ is a \emph{random} probability measure. Recall that the random variables $(H_N)_{N \geq 1}$ and $(\chi_i)_{i \in \N}$ are defined on a common probability space $\Omega$. For a random variable $X : \Omega \to \mfk X$, we let $X^\omega \in \mfk X$ denote the realization of the random variable $X$ on the event $\omega \in \Omega$. Recall the mean magnetization $m_N$ defined in \eqref{e.m_N=} and $\varphi$ defined in \eqref{e.def.phi}. Our next main result is a (quenched) large deviation principle for the random variables $m_N$ under $\la\cdot\ra^{G}_{N}$.


\begin{theorem}[LDP for the magnetization]\label{t.LDP_gen}
    Recall that we assume \eqref{e.convexity}. There exists a full-measure subset $\Omega' \subset \Omega$ such that, for every continuous function $G : \R^d \to \R$ and every $\omega \in \Omega'$, the random variables $m_N$ satisfy a large deviation principle under $\la\cdot\ra^{G,\omega}_{N}$ with rate function $I^G$ defined by 
    \begin{equation*} 
        I^G(m) = -G(m) + \varphi^*(m) +\sup_{m'\in\R^d}\Ll\{G(m') - \varphi^*(m')\Rr\},\quad\forall m \in\R^d.
    \end{equation*}
\end{theorem}


\subsection{Recovering concrete models from the general setting} \label{ss.usual}

The main results of this paper are stated for very general models as introduced in Section~\ref{ss.setting}. Let us explain how to recover results about more usual models from those.  

Given $h:\R^D\times \R^L\to\R^d$ in~\eqref{e.F_N(beta)}, we define $\tilde h:\R^D\times \R^L\to \R^d\times \S^D$ through
\begin{e*}
    \tilde h(\tau,\chi) = \Ll(h(\tau,\chi),\ \tau\tau^\intercal\Rr).
\end{e*}
Henceforth, we isometrically identify $\R^d\times \S^D$ with $\R^{\tilde d}$ for some $\tilde d \in\N$. Given $G:\R^d\to\R$ in~\eqref{e.F_N(beta)}, we introduce $\tilde G:\R^d\times \S^D\to\R$ given by
\begin{e*}
    \tilde G(m,r) =G(m) + \frac{\beta^2}{2}\xi(r).
\end{e*}
Then, we consider $H_N^G$ in~\eqref{e.H^q,x_N(sigma,alpha,chi)=} with $G,h,d$, therein substituted with $\tilde G,\tilde h,\tilde d$ introduced above. Observe that setting $t = \beta^2/2$ and $q = 0$, we have  
\begin{align*}
    H_N^{\tilde G}(\sigma) &= \beta H_N(\sigma) + N \tilde G \left( \frac{1}{N} \sum_{i = 1}^N \tilde h(\sigma_i, \chi_i) \right) - \frac{N \beta^2}{2} \xi \left( \frac{\sigma \sigma^\intercal}{N}\right) \\
    &= \beta H_N(\sigma) + N G \left( \frac{1}{N} \sum_{i  = 1}^N h(\sigma_{\bullet i}, \chi_i) \right).
\end{align*}
In particular, observe that the model described in \eqref{e.simple} is recovered by choosing $D=1$, $d=1$, $L =1$, $h(\sigma,\chi) = \text{sgn}(\sigma \chi)$, $G(m) = m^2$, $\xi(x) = x^2$, and $P_1$ is the uniform measure on $\{-1,1\}$.
In particular observe that thanks to this, Theorem~\ref{t.parisi+ldp.basic} can be recovered from Theorem~\ref{t.F_N(...)} and Theorem~\ref{t.LDP_gen}.

\subsection{Organization of the paper}

In Section~\ref{s.free energy with linear mattis}, we compute the limit free energy $\lim_{N \to +\infty} \bar{F}_N^G$ in the special case where $G$ is of the form $G(m) = x \cdot m$ for some $x \in \R^d$ and prove that this quantity is a continuously differentiable function of $x$. 
In Section~\ref{s.LDP2}, we use the results of the previous section to give a proof of Theorem~\ref{t.F_N(...)} and Theorem~\ref{t.LDP_gen} by relying only on classical large deviation tools.

\subsection{Acknowledgements}

We would like to warmly thank Jean-Christophe Mourrat for many useful inputs and interesting discussions during the conception and writing of this paper. HBC acknowledges funding from the NYU Shanghai Start-Up Fund and support from the NYU–ECNU Institute of Mathematical Sciences at NYU Shanghai.

\section{Identification of the limit free energy with linear Mattis interaction} \label{s.free energy with linear mattis}

The plan to prove Theorem~\ref{t.F_N(...)} is divided into two steps. First, we consider the free energy with $G(m) = x \cdot m$, namely, $ F_N^{m \mapsto x \cdot m}$. In this case, we can think of $ F_N^{m \mapsto x \cdot m}$ as adding some linear random external field to $F_N^0$. For $F_N^{m \mapsto x \cdot m}$, since the linear field decouples easily in the cavity computation, the standard set of tools is effective with minimal modifications. Thus, the following result holds.

\begin{proposition}\label{p.F_N(q,x)}
Recall that we assume \eqref{e.convexity}, for every $x\in\R^d$ we have
\begin{align}\label{e.p.F_N(q,x)}
   \lim_{N\to\infty} \bar{ F}_N^{m \mapsto x \cdot m} = \sup_{p\in\mcl Q_\infty}\Ll\{\psi(q+t\nabla \xi(p);x)-t\int_0^1\theta(p(s))\d s\Rr\}.
\end{align}
\end{proposition}


Observe that at $x = 0$, \eqref{e.p.F_N(q,x)} reduces to the classical Parisi formula as stated in \cite{chenmourrat2023cavity}. At $x \neq 0$, the proof of \eqref{e.p.F_N(q,x)} follows the same arguments as the classical Parisi formula.  For completeness, we sketch the key steps to highlight where convexity and linearity enter. 
After these are clarified, we refer to~\cite{chen2023self} for the details of the two bounds in the setting without a random external field.

The Aizenman--Sims--Starr cavity scheme gives the upper bound by comparing $\bar{ F}_{N+1}^{m \mapsto x \cdot m}$ with $\bar{ F}_N^{m \mapsto x \cdot m}$ and integrating out one spin. This leads to expressions of the form 
\begin{e*}
    (N+1)\bar{ F}^{m \mapsto x \cdot m}_{N+1} - N\bar{ F}^{m \mapsto x \cdot m}_{N} \simeq \psi(q+t\nabla \xi(p);x)-t\int_0^1\theta(p(s))\d s,
\end{e*}
with some $p \in \mcl Q_\infty$ possibly depending on $N$. The lower bound follows from Guerra's Gaussian interpolation \cite{gue03}. Alternatively, one can use the Hamilton--Jacobi supersolution approach of~\cite{mourrat2020nonconvex,mourrat2020free}, with initial condition $q \mapsto \psi(q;x)$.

\begin{proof}[Sketch of proof]
For brevity, let us write $\bar{  F}_N=\bar{  F}_N^{m\mapsto x\cdot m}$.
The interpolation below is written first under the stronger assumption that $\xi$ is convex on all of $\R^{D\times D}$; the final paragraph explains the cited argument that yields the theorem under the stated cone condition~\eqref{e.convexity}.

\smallskip

\textit{Lower bound via Guerra's interpolation~\cite{gue03}.}
Let us display the dependence on $t$ in the Hamiltonian in~\eqref{e.H^q,x_N(sigma,alpha,chi)=} by writing  $H^G_N(\sigma,\alpha;t)$. Fix any $p\in\mcl Q_\infty$ and recall $\theta$ from~\eqref{e.theta=}. We consider an independent Gaussian field $(V(\alpha))_{\alpha\in\supp\fR}$ with covariance
\begin{align*}
    \E \Ll[V(\alpha)V(\alpha')\Rr] = \theta\Ll(p(\alpha\wedge\alpha')\Rr).
\end{align*}
We also take an i.i.d.\ sequence of Gaussian field $(w^{t\nabla\xi(p)}_i(\alpha))_{\alpha\in \supp\fR}$ indexed by $i\in\N$ with covariance  given as in~\eqref{e.Ew^qw^q=} with $q$ therein replaced by $t\nabla\xi(p)$.
For $r\in[0,1]$, we consider the interpolating Hamiltonian
\begin{align*}
    H_N(\sigma,\alpha;r)
    &= H^G_N(\sigma,\alpha;rt)
      +\sqrt{2(1-r)}\sum_{i=1}^Nw^{t\nabla\xi(p)}_i(\alpha)\cdot \sigma_{\bullet i} \\
    &\quad -(1-r)t\nabla\xi(p(1))\cdot\sigma\sigma^\intercal
      +\sqrt{2rNt}V(\alpha)-rNt\theta(p(1)).
\end{align*}
With this Hamiltonian, we define the interpolating free energy
\begin{align*}
    \phi(r) = -\frac{1}{N}\E\log\iint\exp\Ll(H_N(\sigma,\alpha;r)\Rr)\d P_1^{\otimes N}(\sigma)\d \fR(\alpha).
\end{align*}
Then, we have
\begin{align*}
\phi(1)= \bar{ F}_N - \frac{1}{N}\E\log\int\exp\Ll(\sqrt{2Nt}V(\alpha)-Nt\theta(p(1))\Rr)\d\fR(\alpha),
\end{align*}
Here, we used the independence between the Gaussian randomness of $V(\alpha)$ and that of $\sum_{i=1}^N w_i^q(\alpha)\cdot\sigma_{\bullet i}$ in $H^G_N(\sigma,\alpha;rt)$ together with some property of the cascade to decompose $\phi(1)$ into the two terms as in the above display. The term $\frac{1}{N}\E\log\int \cdots, \d \fR(\alpha)$ in the previous display can be computed as
\begin{align*}
&\frac{1}{N}\E\log\int\exp\Ll(\sqrt{2Nt}V(\alpha)\Rr)\d\fR(\alpha) - t\theta(p(1))
\\
&\quad = \Ll(t\theta(p(1))-t\int_0^1\theta(p(s))\d s\Rr) - t\theta(p(1)) = -t\int_0^1\theta(p(s))\d s,
\end{align*}
which appears in~\eqref{e.p.F_N(q,x)}.
The computation of $\frac{1}{N}\E\log\int\exp(\sqrt{2Nt}V(\alpha))\d\fR(\alpha)$ is the same as that for the second term in~\cite[Lemma~3.1]{pan}\footnote{More precisely, one can first consider a step function $p$ and rewrite (e.g., see~\cite[Proposition~4.6]{chenmourrat2023cavity}) $\frac{1}{N}\E\log\int\exp\Ll(\sqrt{2Nt}V(\alpha)\Rr)\d\fR(\alpha)$ in terms of a discrete cascade and carry out the computation as in~\cite[Lemma~3.1]{pan}. For a general $p$, one can proceed by approximation.}.

On the other hand, using $w_i^q(\alpha) + w_i^{t\nabla\xi(p)}(\alpha)\stackrel{\d}{=}w_i^{q+t\nabla\xi(p)}(\alpha)$, we have
\begin{align*}
    \phi(0)=-\frac{1}{N}\E \log\iint\exp\Big(\sum_{i=1}^Nx\cdot h(\sigma_{\bullet i},\chi_i) + \sqrt{2}\sum_{i=1}^N w^{q+t\nabla\xi(p)}_i(\alpha)\cdot \sigma_{\bullet i} 
    \\
    - \big(q+t\nabla\xi(p)\big)(1)\cdot\sigma\sigma^\intercal\Big)\d P_1^{\otimes N}(\sigma)\d\fR(\alpha),
\end{align*}
which is equal to $\psi(q+t\nabla\xi(p);x)$ in~\eqref{e.psi(q;x)=} using the independence of the variables indexed by $i$. Notice that the linearity of the external field was used here. The lower bound in~\eqref{e.p.F_N(q,x)} follows once $\phi(1)\geq \phi(0)$. To show this, write $R_{1,2}=\sigma\sigma'{}^\intercal/N$ and compute
\begin{align}\label{e.d/drphi=}
    \frac{\d}{\d r}\phi(r)=t\E\la\xi(R_{1,2})-\xi(p(\alpha\wedge\alpha'))-\nabla\xi(p(\alpha\wedge\alpha'))\cdot\Ll(R_{1,2}-p(\alpha\wedge\alpha')\Rr)\ra_r,
\end{align}
where we used the standard tool of Gaussian integration by parts and $\la\cdot\ra_r$ is the Gibbs measure associated with $H_N(\sigma,\alpha;r)$. 
Notice that in the above display there are no terms involving self-overlaps
$\sigma \sigma^{\intercal}$ or $\alpha \wedge \alpha = 1$, which might otherwise
arise from Gaussian integration by parts. This is because such terms are exactly
canceled by the derivatives of $-rNt\xi(\frac{\sigma\sigma^\intercal}{N})$, $-(1-r)t\nabla\xi(p(1))\cdot\sigma\sigma^\intercal$, and $-rNt\theta(p(1))$ appearing in the Hamiltonian.
Convexity gives $\xi(a)-\xi(b)-\nabla\xi(b)\cdot(a-b)\geq0$. Hence $\frac{\d}{\d r}\phi(r)\geq 0$, which completes the lower bound.

\smallskip

\textit{Upper bound via Aizenman--Sims--Starr scheme~\cite{aizenman2003extended} and Panchenko's ultrametricity~\cite{pan.aom}.}
Let $(Z(\sigma))_{\sigma\in \R^{D\times N}}$ and $(Y(\sigma))_{\sigma \in \R^{D\times N}}$ be independent centered $\R^{D}$-valued and real-valued Gaussian processes with covariances
\begin{align*}
    \E \Ll[Z(\sigma)Z(\sigma')^\intercal\Rr] = \nabla\xi\Ll(\tfrac{\sigma\sigma'^\intercal}{N}\Rr)\quad\text{and}\quad \E \Ll[Y(\sigma)Y(\sigma')\Rr] =\theta\Ll(\tfrac{\sigma\sigma'^\intercal}{N}\Rr).
\end{align*}
We can compute that
\begin{align*}
    (N+1)\bar F_{N+1}- N\bar F_{N} = A_N^{(1)} - A_N^{(2)} + o(1),
\end{align*}
with an error term $o(1)$ that vanishes as $N\to\infty$.
Here,
\begin{align*}
    A_N^{(1)} = -\E \log\bigg\langle \int \exp\Big(&\sqrt{2t}Z(\sigma)\cdot\tau
    +\sqrt{2}w^q_{N+1}(\alpha)\cdot \tau \\
    &-t\nabla\xi\Ll(\tfrac{\sigma\sigma^\intercal}{N}\Rr)\cdot \tau\tau^\intercal
    -q(1)\cdot \tau\tau^\intercal \\
    +x\cdot h(\tau,\chi_{N+1})\Big)  \d P_1(\tau)\bigg\rangle_N,
\end{align*}
and
\begin{align*}
    A_N^{(2)}=-\E \log \la \exp\Ll(\sqrt{2t}Y(\sigma)
    - t\theta\Ll(\tfrac{\sigma\sigma^\intercal}{N}\Rr)\Rr) \ra_N.
\end{align*}
where $\tau=\sigma_{\bullet N+1}$ is called the cavity spin and $\la\cdot\ra_N$ is the Gibbs measure associated with $\bar F_N$. Notice that $w^q_{N+1}(\alpha)$ and $h(\tau,\chi_{N+1})$ have randomness independent from that of $\la\cdot\ra_N$. We used the linearity $(N+1)x\cdot m_{N+1} = Nx\cdot m_N + x\cdot h(\tau,\chi_{N+1})$, where $Nx\cdot m_N$ is absorbed in $\la\cdot\ra_N$ and the last term is displayed above.

We can see $\limsup_{N\to\infty} \bar F_N\leq \limsup_{N\to\infty}A^{(1)}_N - A^{(2)}_N$. Heuristically, by passing  to a subsequence, we may assume that the distribution of the spin overlap $\frac{\sigma\sigma'^\intercal}{N}$ under $\E\la\cdot\ra_N$ converges to that of $p(U)$ for some $p\in\mcl Q_\infty$ where $U$ is a uniform random variable over $[0,1]$. Also, we can \textit{synchronize} the spin overlap with the cascade overlap of $\alpha$ as spelled out in~\cite[Theorem~4]{pan.vec}. Panchenko's ultrametricity result along with Ghirlanda--Guerra's identities tells us that this is enough to identify the limit of $A^{(1)}_N$ and $A^{(2)}_N$, which are $\psi(q+t\nabla \xi(p);x)$ and $t\int_0^1\theta(p(s))\d s$, respectively. 

This completes the upper bound modulo the standard perturbation and synchronization arguments needed to justify the subsequential overlap description.

\textit{Relaxation of convexity of $\xi$.}
In view of the above proof, the convexity of $\xi$ on $\mathbb{R}^{D \times D}$
is only used to ensure that the right-hand side of~\eqref{e.d/drphi=} is
nonnegative.  
When $D=1$, Talagrand's positivity principle allows one to relax this assumption
and require $\xi$ to be convex only on $\mathbb{R}_{+}$.  
When $D>1$, one can verify that~\cite[Proposition~8.1]{chenmourrat2023cavity}
remains valid in the presence of the random external field $N x \cdot m_N$,
which yields Proposition~\ref{p.F_N(q,x)} with $\xi$ assumed to be convex only
on $\S^{D}_{+}$. 

Indeed, the proof of~\cite[Proposition~8.1]{chenmourrat2023cavity} is based on cavity
computations (similar to those used in the above proof of the upper bound) developed
in~\cite{chenmourrat2023cavity}, together with an input from~\cite{mourrat2020free}
that provides a lower bound (in place of Guerra's bound) via the solution of a
relevant Hamilton--Jacobi equation. The variational representation of this solution
is precisely the right-hand side of~\eqref{e.p.F_N(q,x)}. In both~\cite{chenmourrat2023cavity} and~\cite{mourrat2020free}, the analysis is
carried out for vector spin glass models without a random external field. However,
as illustrated in the sketch above, the inclusion of such a field does not require
any substantial modification of the arguments.
\end{proof}

Recall that in \eqref{e.def.phi} we have defined $\varphi(x) = \lim_{N\to\infty} \bar{ F}^{m \mapsto x \cdot m}_N$. It follows from Proposition~\ref{p.F_N(q,x)} that 
\begin{equation} \label{e.varphi=sup}
    \varphi(x) = \sup_{p\in\mcl Q_\infty}\Ll\{\psi(q+t\nabla \xi(p);x)-t\int_0^1\theta(p(s))\d s\Rr\}.
\end{equation}
As explained in the following lemma, $\varphi$ is a well-behaved function of $x$. This follows from similar arguments as the ones for~\cite[Lemma~2.2]{chen2024conventional}.
%
\begin{lemma}\label{l.varphi}
The function $\varphi$ is Lipschitz, concave, and continuously differentiable.
\end{lemma}

\begin{proof}
We denote by $\msc P(p,x)$ the expression inside $\sup_{p\in\mcl Q_\infty}\{\cdots\}$ in~\eqref{e.varphi=sup}, so that 
\begin{e*}
    \varphi(x)=\sup_{p\in\mcl Q_\infty}\msc P(p,x).
\end{e*}
Furthermore, we have $\nabla_x \bar F^{m \mapsto x \cdot m}_N = -\E \la  m_N\ra_{N}^{m \mapsto x \cdot m}$, which is bounded uniformly in $N$. For any $v\in\R^d$, we can compute 
\begin{e*}
    (v\cdot\nabla_x)^2\bar F^{m \mapsto x \cdot m}_N = -N\E\Ll[\la (v\cdot m_N)^2\ra_{N}^{m \mapsto x \cdot m}-\left(\la v\cdot m_N\ra_{N}^{m \mapsto x \cdot m}{}\right)^2\Rr]\leq 0.
\end{e*}
Therefore, $\varphi$ is Lipschitz and concave.

It is classical that a finite concave function that is differentiable everywhere is continuously differentiable (apply, for example, \cite[Theorem~25.5]{rockafellar1970convex} to $-\varphi$). Hence, it only remains to show that $\varphi$ is differentiable everywhere.
Fix any $x \in \R^d$. 
A vector $a\in \R^d$ is said to be a supergradient of $\varphi$ at $x$ when for every $y\in\R^d$,
\begin{e*}
    \varphi(y)-\varphi(x)\leq a\cdot(y-x).
\end{e*} 
Since $\varphi$ is concave, to show that $\varphi$ is differentiable, it suffices to show that at each point, $\varphi$ admits at most one supergradient. 
For each $\eps>0$, we choose $p_\eps$ to satisfy
\begin{align}\label{e.P(pi_eps,x)<P(x)+1/n}
    \msc P(p_\eps,x) \geq \varphi(x) -\eps.
\end{align}
Fix any $y\in \R^d$ and let $r\in (0,1]$.
Let $a$ be any supergradient of $\varphi$ at $x$. Using its definition, the fact that $\varphi$ is a supremum, and \eqref{e.P(pi_eps,x)<P(x)+1/n}, we have
\begin{align*}
    y\cdot a &\geq \frac{\varphi(x+ry) - \varphi(x)}{r} \geq \frac{\msc P(p_\eps,x+ry)-\msc P(p_\eps,x)-\eps}{r},
    \\
    y\cdot a &\leq \frac{\varphi(x) - \varphi(x-ry)}{r} \leq \frac{\msc P(p_\eps,x)-\msc P(p_\eps,x-ry)+\eps}{r}.
\end{align*}
Using the expression in~\eqref{e.psi(q;x)=}, we can compute derivatives of $\msc P(p_\eps,\cdot)$ at $x$ and verify that they are bounded.
Hence, we can use Taylor's expansion of $\msc P(p_\eps,\cdot)$ at $x$ in the above display to get for some constant $C$ independent of $r$ and $\eps$, we have 
\begin{align*}
    \Ll|y\cdot a - y\cdot \nabla\msc P(p_\eps,x)\Rr| \leq C r+\eps r^{-1}.
\end{align*}
Setting $r= \sqrt{\eps}$ and sending $\eps\to0$, we see that every supergradient $a$ is approximated by the same quantity $\nabla\msc P(p_\eps,x)$, for the same choice of $p_\eps$. Thus, for any two supergradients $a$ and $a'$, $y\cdot(a-a')=0$. Since $y$ is arbitrary, $a=a'$, so $\varphi$ is differentiable at $x$. As explained earlier, this implies that $\varphi$ is continuously differentiable.
\end{proof}

\begin{remark}
    Since $\varphi$ is Lipschitz we have $\varphi^*(m) = +\infty$ for $|m| > L$ where $L$ is a Lipschitz constant of $\varphi$.
\end{remark}

\section{Proofs of main results} \label{s.LDP2}

In this section, we provide a proof of our main results Theorem~\ref{t.F_N(...)} and Theorem~\ref{t.LDP_gen}. We will rely on the differentiability of $\varphi(x)=  \lim_{N \to +\infty} \bar{ F}_N^{m \mapsto x \cdot m}$ as well as classical large deviation tools. 
Namely we will use the  Gärtner--Ellis theorem \cite[Theorem~2.3.6]{dembo2009large}, Varadhan's lemma \cite[Theorem~4.3.1]{dembo2009large}, and Bryc's inverse Varadhan lemma \cite[Theorem~4.4.13]{dembo2009large}.


\begin{lemma}\label{l.LDP_<>}
There is a full-measure subset $\Omega' \subset \Omega$ such that, for every $\omega\in\Omega'$, $m_N$ under $\la\cdot\ra^{0, \omega}_{N}$ satisfies large deviation principle with rate function $I$ defined by
\begin{align}\label{e.I_x=}
    I(m) = \varphi^*(m) - \varphi(0),\quad\forall m \in\R^d.
\end{align}
\end{lemma}

\begin{proof}
Let $\Omega' \subset \Omega$ be the full-measure subset such that for every $\omega\in\Omega'$, we have for all $y\in \Q^d$,
\begin{equation*}
    \lim_{N\to\infty} F^{m \mapsto m \cdot y, \omega}_N = \varphi(y).
\end{equation*}
Henceforth, we fix $\omega\in\Omega'$ and define for each $y\in\R^d$,
\begin{e*}
    \Lambda_N(y)= \frac{1}{N}\log \la e^{Ny\cdot m_N}\ra^{0,\omega}_{N}.
\end{e*}
Notice that $\Lambda_N(y) = -F_N^{m \mapsto y \cdot m,\omega} + F_N^{0,\omega}$. Hence, we have  for every $y\in \Q^d$
\begin{e*}
    \lim_{N\to\infty}\Lambda_N(y)=-\varphi(y)+\varphi(0).
\end{e*}
Since $\Lambda_N$ is Lipschitz uniformly in $N$ and $\varphi$ is also Lipschitz, the previous display in fact holds for every $y \in \R^d$. Finally, since by Lemma~\ref{l.varphi}, $\varphi$ is differentiable everywhere, the Gärtner--Ellis theorem (e.g., see~\cite[Theorem~2.3.6]{dembo2009large}) guarantees that $m_N$ satisfies a large deviation principle under $\la\cdot\ra^{0,\omega}_{N}$ with rate function 
\begin{e*}
    I(m) = \sup_{y \in \R^d} \left\{ y \cdot m - \lim_{N \to +\infty} \Lambda_N(y) \right\} = \varphi^*(m) - \varphi(0).
\end{e*}
\end{proof}

\begin{lemma}\label{l.log<e^G>}
For every continuous function $G:\R^\d\to\R$ and every $\omega\in\Omega'$ (given as in Lemma~\ref{l.LDP_<>}), we have 
\begin{e*}
    \lim_{N\to\infty}\frac{1}{N}\log\la e^{NG(m_N)}\ra^{0,\omega}_{N}=\sup_{m\in\R^d}\Ll\{G(m)-I(m)\Rr\},
\end{e*}
with $I$ given as in~\eqref{e.I_x=}.
\end{lemma}

\begin{proof}
Since $m_N$ is bounded uniformly in $N$, so is $G(m_N)$, which allows us to apply Varadhan's lemma (e.g., see~\cite[Theorem~4.3.1]{dembo2009large}) to deduce the desired result from the large deviation principle in Lemma~\ref{l.LDP_<>}.
\end{proof}

Using Lemma~\ref{l.log<e^G>} in conjunction with Bryc's inverse Varadhan lemma, we obtain a proof of Theorem~\ref{t.LDP_gen}.

\begin{proof}[Proof of Theorem~\ref{t.LDP_gen}]
Fix any $G$ and let $\Gamma:\R^d\to\R$ be any bounded continuous function. Since 
\begin{e*}
    \frac{1}{N}\log\la e^{N\Gamma(m_N)}\ra^{G,\omega}_{N} = \frac{1}{N}\log\la e^{N(G+\Gamma)(m_N)}\ra^{0,\omega}_{N} - \frac{1}{N} \log\la e^{NG(m_N)}\ra^{0,\omega}_{N},
\end{e*}
Lemma~\ref{l.log<e^G>} implies that for every $\omega \in \Omega'$ we have 
\begin{align*}
    \lim_{N\to\infty}\frac{1}{N}\log\la e^{N\Gamma(m_N)}\ra^{G,\omega}_{N}
    &=\sup\Ll\{G+ \Gamma-\varphi^* \Rr\} - \sup\Ll\{G-\varphi^* \Rr\} \\
    &= \sup \{\Gamma - I^G\}.
\end{align*}
Since this holds for any bounded continuous $\Gamma$, we can apply Bryc's inverse Varadhan lemma (e.g., see the ``$\Longleftarrow$'' direction in~\cite[Theorem~4.4.13]{dembo2009large}) which guarantees that $m_N$ satisfies a large deviation principle under $\la\cdot\ra^{G,\omega}_{N}$ with rate function $I^G$.
\end{proof}


From Lemma~\ref{l.log<e^G>}, we can derive the Parisi-type formula for $\lim_{N \to +\infty}  F_N^G$.

\begin{proof}[Proof of Theorem~\ref{t.F_N(...)}]
Notice that 
\begin{e*}
    \frac{1}{N}\log\la e^{NG(m_N)}\ra^{0,\omega}_{N}= -F^{G,\omega}_N + F^{0,\omega}_N.
\end{e*}
Using Lemma~\ref{l.log<e^G>}, we get that for every $\omega \in \Omega'$
\begin{e*}
    \lim_{N\to\infty}F^{G,\omega}_N = \inf_{m \in \R^d} \left\{\varphi^*(m) - G(m) \right\}.
\end{e*}
Finally, using the variational formula for $\varphi(x)$ of Proposition~\ref{p.F_N(q,x)}, we observe that the right-hand side in the previous display is equal to the right-hand side in \eqref{e.t.gen}.
\end{proof}


\small
\bibliographystyle{plain}
\bibliography{ref}

\end{document}